\documentclass[a4paper,12pt]{article}
\usepackage{hyperref}
\usepackage[english]{babel}
\usepackage[T1]{fontenc}
\usepackage[utf8]{inputenc}

\usepackage{amsfonts}
\usepackage{amssymb}
\usepackage{amsmath}
\usepackage{graphicx}
\usepackage{cite}
\usepackage{epsfig}
\usepackage{psfrag}
\usepackage{url}
\usepackage{amsthm}
\usepackage{enumerate}
\usepackage{comment}

\newtheorem{proposition}{Proposition}[section]
\newtheorem{definition}{Definition}[section] 
  
\newtheorem{corollary}{Corollary}[section] 
\newtheorem{theorem}{Theorem}[section] 
\newtheorem{remark}{Remark}[section]  
\newtheorem{example}{Example}[section]

\newcommand{\dex}[1]{\operatorname{d}\!}
\def\g{\mathfrak{g}}
\def\gl{\mathfrak{gl}}

\date{}

\title{\bf Pre-Lie Structures for Semisimple Lie Algebras}

\author{Xerxes D. Arsiwalla$^{1, }$\footnote{\url{x.d.arsiwalla@gmail.com}}  { }  {  } Fernando Olivie M{\'e}ndez M{\'e}ndez$^{2, }$\footnote{\url{jxvm0952@leeds.ac.uk}}    \\  
{}  \\
{\it \small $^{1}$Wolfram Institute for Computational Foundations of Science,  USA}\\  
{\it \small $^{2}$University of Leeds, UK}   
}

\begin{document}
\maketitle

\begin{abstract}

We address the problem of admissibility of pre-Lie structures associated with a given Lie algebra, particularly, semisimple Lie algebras over ${\mathbb C}$. Such structures are collectively referred to as Lie-admissible algebras, which are a class of nonassociative algebras such that the commutator bracket over these algebras satisfies the Jacobi identity. Among the five classes of nonassociative Lie-admissible algebras, left-symmetric algebras (LSAs) and right-symmetric algebras (RSAs), are known to be non-admissible by semisimple Lie algebras of finite dimension $n \geq 3$. Here, we examine the remaining classes starting with those corresponding to the subgroup generated by permutations of order 2: $(1 \; 3)$. These appear in the literature as anti-flexible algebras (AFAs). We discuss properties of AFAs and provide examples of finite-dimensional representations. AFAs geometrically correspond to richer structures than the flat torsion-free affine connections associated with left-symmetric algebras (LSAs) or right-symmetric algebras (RSAs). We compute Lie-admissibility criteria for AFAs and determine a few simple solution classes. Not surprisingly, solvable Lie algebras admit AFAs. Concerning semisimple ones, we report an explicit counterexample demonstrating an AFA admissible by ${\mathfrak sl(2, \, {\mathbb C})}$. We then discuss the remaining two classes of nonassociative Lie-admissible algebras, the $A_3$-associative and $S_3$-associative types. Finally, we prove that  $S_3$-associative algebras are universal pre-Lie structures for any Lie algebra over ${\mathbb C}$, including semisimple ones. 

\end{abstract}

\vspace{2pc}
  
\clearpage

\tableofcontents

\section{Introduction}

The earliest examples of Lie-admissible algebras originated in the work of Cayley on rooted tree algebras \cite{cayley1857xxviii}, and later in Koszul's and Vinberg's investigations on affine structures on manifolds \cite{koszul1961domaines, vinberg1963convex}. Then, almost fourteen years later, it was Milnor's question concerning the admissibility of affine structures on manifolds associated to solvable Lie groups, that regenerated excitement in the study of pre-Lie structures and their associated Lie-admissible algebras for various types of Lie groups    \cite{milnor1977fundamental}.     

Lie-admissible algebras include five classes of nonassociative algebras (and one class of associative algebras) that can be thought of as precursors to Lie algebras \cite{santilli1968introduction, remm2002operades, goze2004lie}. 
A pre-Lie structure over a field ${\mathbb F}$ is a Lie-admissible algebra $\left( {\cal A}, \cdot  \right)$ with a bilinear product $(x , y) \to x \cdot y$ such that the commutator $[ x , y ] = x \cdot y - y \cdot x$ defines a Lie algebra. Pre-Lie structures are closely related to affine structures on manifolds associated to Lie groups. The most frequently discussed examples of Lie-admissible algebras are Left-Symmetric Algebras (LSAs) and Right-Symmetric Algebras (RSAs). LSAs are often referred to as Vinberg algebras. RSAs are sometimes called pre-Lie algebras (though some authors use the term "pre-Lie algebras" more generally to include both of these classes). LSAs define flat torsion-free left-invariant affine structures, and thus carry a geometric interpretation. RSAs are the "opposite algebras" of LSAs. The correspondence between these algebraic structures and the underlying geometry can be exploited for extracting geometric data from  algebraic properties of associated pre-Lie structures on the manifold.  

Lie-admissible algebras (except for the associative class) have an interesting nonassociative structure. The triple product, also known as the "associator", defined by $(x , y , z) = (x \cdot y) \cdot z - x \cdot (y \cdot z)$, is generally non-vanishing for given $x, y, z \in {\cal A}$.  In particular,  LSAs and RSAs are defined by specific relations among their associators. An LSA satisfies $(x , y , z) = (y , x , z)$; whereas an RSA satisfies $(x , y , z) = (x , z , y)$ for all $x, y, z \in {\cal A}$. It follows from these definitions that LSAs and RSAs are Lie-admissible. As mentioned, LSAs and RSAs are opposite algebras with respect to each other. This means that transforming the product $x \cdot y \to y \cdot x$ maps an LSA to an RSA, and vice versa. An excellent review of these algebras can be found in \cite{burde2006left}.   

In addition to LSAs and RSAs, other classes of Lie-admissible algebras have also been investigated \cite{goze2004lie}. These are in one-to-one correspondence with the subgroups of the symmetric group $S_3$. Among those, we will start by examining Lie-admissible algebras corresponding to the subgroup generated by permutations of the 2-cycle $(1 \; 3)$. These are also called Anti-Flexible Algebras (AFAs), which were originally investigated for their own sake as generalizations of the flexible law \cite{rodabaugh1965generalization, rodabaugh1967semisimple, rodabaugh1972antiflexible, bhandari1972classification, myung1978nonflexible}. In the context of Lie-admissible algebras, they have also been referred to as Center-Symmetric Algebras \cite{hounkonnou2016center}. In this work, we discuss various properties of AFAs and provide examples of finite-dimensional representations. Geometrically, AFAs correspond to richer structures than flat torsion-free affine connections associated to LSAs or RSAs. This follows directly from the anti-flexible / center-symmetric condition that is satisfied by left- and right-insertion operators associated with AFAs. We then move to the remaining two classes of Lie-admissible algebras, namely, the $A_3$-associative and $S_3$-associative algebras and their finite-dimensional representations.

The main focus of our work concerns pre-Lie structures admissible by semisimple Lie algebras over ${\mathbb C}$. It is known that there are no admissible LSAs or RSAs for such Lie algebras for finite dimensions $n \geq 3$   \cite{helmstetter1979radical, burde1994left, burde1998simple, burde1999left,  burde2006left}. On the other hand, solvable Lie algebras do admit LSAs and RSAs. What happens in the case of AFAs? To address this, we first compute Lie-admissibility criteria for AFAs and determine a few simple  solution classes. Given any Lie algebra, we can check whether a particular solution class of AFAs is admissible or not. It turns out solvable Lie algebras do admit AFAs. However, this should not be surprising, since AFAs are labeled via the same conjugacy class of the group $S_3$ as LSAs and RSAs. For this reason, one might at first expect that semisimple Lie algebras too would not admit AFAs. But this is where things get interesting. We report an explicit counterexample which demonstrates an AFA admissible by ${\mathfrak sl(2, \, {\mathbb C})}$. We also show an example of an $A_3$-associative algebra admissible by ${\mathfrak su(2)}$. In fact, we prove that not only semisimple Lie algebras, but every Lie algebra (over ${\mathbb C}$) admits a pre-Lie structure (often of multiple Lie-admissible classes).

The outline of this paper is as follows. Most of Section 2 covers preliminary definitions and results: we define AFAs, prove various results related to these algebras, discuss the geometric interpretation of AFAs, and provide examples of finite-dimensional AFAs. In Section 3 we discuss Lie-admissibility criteria for AFAs, and some of their solution classes. We show how these conditions can be used to compute admissible AFAs for various Lie algebras. In Section 4 we discuss $A_3$- and $S_3$-associative algebras, along with examples of their finite-dimensional realizations. Section 5 discusses different classes of pre-Lie structures admissible by semisimple Lie algebras. Finally, in Section 6 we conclude with closing remarks and open questions.

\section{Anti-Flexible Algebras (AFAs) }

Let $(A,\cdot)$ be an algebra that is a $\mathbb{F}$-vector space together with a bilinear product $\cdot: A \times A \longrightarrow A$. Consider a permutation $\sigma \in  S_3$ together with the following condition on the associators: 
\begin{equation} \label{eq:associator_sigma}
    (x_1,x_2,x_3) \, = \, (x_{\sigma(1)},x_{\sigma(2)},x_{\sigma(3)}) 
\end{equation}
for any $ x_1,x_2,x_3  \in A$.  Notice that when the condition in eq.~\eqref{eq:associator_sigma} is satisfied for the permutation $\sigma = ( 1 \hspace{1mm} 2 )$ (respectively, $\sigma = (2 \hspace{1mm} 3)$) the algebra in question is the well-known left-symmetric algebra (respectively, right-symmetric algebra). Thus, the action of non-trivial elements of $S_3$ on the associator leads to definitions of nonassociative algebras. A particularly interesting case for our discussion here will be the action of the permutation $\sigma = (1 \hspace{1mm} 3)$ on the associator.  The complete classification of algebras resulting from the action of $S_3$ on the associator was first reported in \cite{goze2004lie}. What was recognized by those authors was that each sub-group of the symmetric group $S_3$ corresponds to distinct classes of (non)associative algebras, each of which turn out to be Lie-admissible (i.e., their commutator satisfies the Jacobi identity). These classes are defined by the condition \cite{goze2004lie}:  
\begin{eqnarray}
    \sum_{\sigma \
\in \ G}{{\rm \ sgn}(\sigma) \, (x_{\sigma(1)}, x_{\sigma(2)}, x_{\sigma(3)})}  = 0 
\label{subgsum}
\end{eqnarray}
where $G$ is a given subgroup of $S_3$, and ${\rm \ sgn}(\sigma)$ denotes the signature of the permutation $\sigma \in G$. When $G$ is just the identity, this corresponds to associative algebras;  whereas the other five subgroups of $S_3$ each yield distinct classes of nonassociative algebras. As mentioned, the order-2 subgroup that includes the element $\sigma = (1 \hspace{1mm} 3)$ will be particularly relevant to us.

\begin{definition}
We define an \textbf{Anti-Flexible Algebra (AFA)} as an algebra $\left( {\cal A}, \cdot  \right)$ over a field ${\mathbb F}$ equipped with a bilinear product  $(x , y) \mapsto x \cdot y$ such that the associator satisfies 
\begin{eqnarray}
    (x , y , z) = (z , y , x)
\label{def2.1}
\end{eqnarray}
for all $x, y, z \in {\cal A}$. 
\end{definition}
 
Using this definition, it is straightforward to see that an AFA is Lie-admissible. This follows using the identity (in the literature, this is also known as the Akivis identity    \cite{akivis1976local}): 
\begin{equation} \label{eq:jacobi_identity}
    [ x_ 1, [ x_ 2,x_ 3] ] + [ x_ 2, [ x_ 3, x_ 1] ] + [ x_ 3, [ x_ 1, x_ 2] ]  \, = \,  \sum_{\sigma \
\in \ S_ 3}{{\rm \ sgn}(\sigma) \, (x_{\sigma(1)}, x_{\sigma(2)}, x_{\sigma(3)})} 
\end{equation}
where ${\rm \ sgn}(\sigma)$ denotes the signature of the permutation $\sigma$, and the commutator is defined in the usual way, i.e., $[ x , y ] := x \cdot y - y \cdot x$. 
While the above identity holds in any algebra defined by the commutator product, it is the vanishing of the right-hand side (RHS) that reduces eq.~\eqref{eq:jacobi_identity} to the Jacobi identity, thus realizing a Lie algebra. As can be readily checked, in the case of an AFA, the sum of associators on the RHS of eq.~\eqref{eq:jacobi_identity}  indeed vanishes.  


Similarly to left-symmetric structures, one can also define anti-flexible structures as follows: 
\begin{definition}
An \textbf{Anti-Flexible Structure} in a Lie algebra \(\mathfrak{g}\)  is a product \((x,y) \mapsto x \cdot y\) that satisfies the following conditions: 
\begin{enumerate}
\item \((\mathfrak{g}, \cdot) \) is an anti-flexible algebra (AFA);
\item \([x,y]_{\mathfrak{g}} = x \cdot y - y \cdot x\) for all \(x,y \in \mathfrak{g}\).
\end{enumerate}   
 In this case, we say that $\mathfrak{g}$ admits an anti-flexible structure.               
\end{definition}

\subsection{Properties }

Now, let us discuss some properties of AFAs. Unlike an LSA or RSA, an AFA is its own opposite algebra. That is,  $x \cdot y \to y \cdot x$  leaves the AFA associator identity invariant. More generally, the following holds: 

\begin{proposition}
If $({\cal A}, \, \cdot)$ is an algebra that satisfies eq.~\eqref{eq:associator_sigma} with respect to the permutation $\sigma$, and we define the product
 \begin{equation*}
    \circ: {\cal A} \times {\cal A} \longrightarrow {\cal A}  \quad \textrm{as} \quad (x,y) \mapsto y \cdot x
\end{equation*}
then, the algebra $({\cal A},\, \circ)$ satisfies eq.~\eqref{eq:associator_sigma} with respect to the permutation $\tau = (1 \hspace{1mm} 3) \, \sigma \, (1 \hspace{1mm} 3)$, i.e., with respect to the permutation resulting from the conjugation of $\sigma$ by $(1 \hspace{1mm} 3)$.
\end{proposition}
\begin{proof}
Denote the associator in $({\cal A}, \, \circ)$ by $[\cdot \, , \, \cdot \, , \, \cdot]$. Then, the claim follows from the following:  
\begin{eqnarray}   
 [x_1,x_2,x_3] &=& - (x_{\sigma^{\prime}(1)}, x_{\sigma^{\prime}(2)}, x_{\sigma^{\prime}(3)}) = -(x_{\sigma \, \sigma^{\prime}(1)}, x_{\sigma \, \sigma^{\prime}(2)}, x_{\sigma \, \sigma^{\prime}(3)})   \nonumber  \\   
 &=&  [x_{\sigma^{\prime} \, \sigma \, \sigma^{\prime}(1)}, \, x_{\sigma^{\prime} \, \sigma \, \sigma^{\prime}(2)}, \, x_{\sigma^{\prime} \, \sigma \, \sigma^{\prime}(3)}]
\end{eqnarray}
where $\sigma^{\prime}$ denotes the permutation $(1 \hspace{1mm} 3)$. 
\end{proof}
As per Proposition 2.1 above, AFAs are their own opposite algebras. This contrasts with LSAs and RSAs, which are mutually opposite to each other.

We find the following relations between LSAs, RSAs, and AFAs:
\begin{proposition}
Let $\left( {\cal A}, \cdot  \right)$ be an algebra over a field ${\mathbb F}$ (or more generally, over a ring). If ${\cal A}$ satisfies the conditions of any two: a LSA, RSA or AFA, then it also satisfies the third.  
\end{proposition}
\begin{proof}
This follows from eq.~\eqref{eq:jacobi_identity}. Let $\tau_1$ and $\tau_2$ be any two distinct transpositions in $S_3$, and let $\tau_3$ be the remaining of the three transpositions. Suppose that eq.~\eqref{eq:associator_sigma} holds with respect to $\tau_1$ and $\tau_2$. Furthermore, let $\sigma_1$ and $\sigma_2$ be any of the two distinct $3$-cycles in $S_3$. Then, by equation \eqref{eq:jacobi_identity} we have: 
\begin{align*}
0 = &\  (x_1, x_2, x_3)  - (x_{\tau_3(1)}, x_{\tau_3(2)}, x_{\tau_3(3)})    \\
   &\  + (x_{\sigma_1(1)}, x_{\sigma_1(2)}, x_{\sigma_1(3)}) - (x_{\tau_1 \sigma_1 (1)}, x_{\tau_1 \sigma_1 (2)}, x_{\tau_1 \sigma_1 (3)})   \\
   &\  + (x_{\sigma_2(1)}, x_{\sigma_2(2)}, x_{\sigma_2(3)}) - (x_{\tau_2 \sigma_2 (1)}, x_{\tau_2 \sigma_2 (2)}, x_{\tau_2 \sigma_2 (3)}) 
\end{align*}
Here, the vanishing LHS is a consequence of eq.~\eqref{eq:associator_sigma} with respect to either one of $\tau_1$ or $\tau_2$. In the RHS, the permutations $\tau_1$ and $\tau_2$, without loss of generality, can be expressed as $\tau_2 \sigma_2$ and $\tau_1 \sigma_1$, respectively.  Hence, the third and fourth summands in the above equation cancel each other out (following eq.~\eqref{eq:associator_sigma} for $\tau_1$), as do the fifth and sixth summands (following eq.~\eqref{eq:associator_sigma} for $\tau_2$). We thus get   
\[  0 =  (x_1, x_2, x_3) - (x_{\tau_3(1)}, x_{\tau_3(2)}, x_{\tau_3(3)})  \]
which proves the result.  
\end{proof}
Notice the analogy here to the three alternate algebras: left, right, and flexible alternate algebras, respectively. There again, the existence of any two alternate algebras implies the third.



Furthermore, just like the product governing LSAs and RSAs, we find that AFAs too are neither commutative nor anticommutative. We prove this for all three algebras as follows.  
\begin{proposition} \label{prop: commutative_symmetric_algebra}
If $\left( {\cal A}, \cdot  \right)$ is either a commutative or an anticommutative algebra, then $\left( {\cal A}, \cdot  \right)$ does not satisfy eq.~\eqref{eq:associator_sigma} for an LSA, RSA or AFA. 
\end{proposition}
\begin{proof}
The associators of any commutative or anticommutative algebra satisfy the identity 
\begin{eqnarray}
(x_1, x_2, x_3) =  - (x_{\sigma^{\prime}(1)}, x_{\sigma^{\prime}(2)}, x_{\sigma^{\prime}(3)})  \label{anticommid}      
\end{eqnarray}
obtained by commuting or anticommuting the product $x \cdot y$, respectively; and where $\sigma^{\prime}$ refers to the permutation $( 1 \; 3 )$. Now suppose that the LHS of the above identity satisfies eq.~\eqref{eq:associator_sigma} for an AFA, that is, for $\sigma = ( 1 \; 3 )$. Then, all the associators of this algebra vanish, making it an associative algebra. Hence, $\left( {\cal A}, \cdot  \right)$ is not an AFA. 

Now suppose $\left( {\cal A}, \cdot  \right)$ satisfies  eq.~\eqref{eq:associator_sigma} for an LSA. Applying eq.~\eqref{anticommid} to both sides of the LSA identity leads to
\[  - (x_{\sigma^{\prime}(1)}, x_{\sigma^{\prime}(2)}, x_{\sigma^{\prime}(3)}) =  - (x_{\sigma^{\prime}(2)}, x_{\sigma^{\prime}(1)}, x_{\sigma^{\prime}(3)}) \]
which, in fact, yields an RSA (following the action of $\sigma^{\prime}$). But we have seen earlier from Proposition 2.2 that if $\left( {\cal A}, \cdot  \right)$ satisfies both, the LSA and RSA condition, then it also defines an AFA. This leads to a contradiction as we know from the argument following eq.~\eqref{anticommid} that $\left( {\cal A}, \cdot  \right)$ is not an AFA. This implies that the above assumption that $\left( {\cal A}, \cdot  \right)$ satisfies a LSA does not hold. A similar reasoning suggests $\left( {\cal A}, \cdot  \right)$ is not a RSA. 
\end{proof}

It is well-known that in the case of LSA structures, a product $x \cdot y$ defines an LSA structure if and only if the left-insertion operator $L:\g \rightarrow \gl(\g)$ defined by $L_x(y):=x \cdot y$ is a representation and $\operatorname{Id}_\g$ is a $1$-cocycle of $L$. In the case of AFA structures, we have the following characterization result (see also \cite{hounkonnou2016center}). 

\begin{proposition} \label{prop:csa_equivalences}
    Let $\g$ be a Lie algebra equipped with a bilinear product $(x,y) \mapsto x \cdot y$. Let $L, R \colon \g \to \gl(\g)$ denote the left and right multiplication operators defined by $L_x(y) = x \cdot y$ and $R_x(y) = y \cdot x$, and let $S_x = L_x + R_x$ be the symmetric multiplication operator. The following conditions are equivalent:
    \begin{enumerate}
        \item The product $x\cdot y$ defines an AFA structure on $\g$.
        
        \item The left and right multiplication operators satisfy the commutation relation
        \begin{eqnarray}
            [L_x, R_y] = [L_y, R_x] \qquad \text{for all } x, y \in \g. 
            \label{csacon}
        \end{eqnarray}
        
        \item The symmetric multiplication operators satisfy the condition
        \begin{eqnarray}
            [S_x, S_y] = \operatorname{ad}_{[x,y]} \qquad \text{for all } x, y \in \g,
        \end{eqnarray}
        where $\operatorname{ad}:\g \rightarrow \gl(\g)$ is the adjoint representation of $\g$. 
    \end{enumerate}
\end{proposition}
\begin{proof}
\mbox{}
    \begin{enumerate}
        \item[$1) \Rightarrow 2):$] Let $x,y \in \g$ be fixed. Given any $z \in \g$, we have that $[L_x,R_y](z)=(L_x \circ R_y-R_y \circ L_x)(z)=x \cdot (z \cdot y)-(x \cdot z) \cdot y=-(x,z,y)$. Analogously, $[L_y,R_x](z)=-(y,z,x)$. The AFA condition implies $(x,z,y)=(y,z,x)$ and consequently $[L_x,R_y]=[L_y,R_x]$. 
        \item[$2) \Rightarrow 3):$] The bilinearity of the bracket implies that $[S_x,S_y]=[L_x+R_x,L_y+R_y]=[L_x,L_y]+[L_x,R_y]+[R_x,L_y]+[R_x,R_y]$. Now, the assumption $[L_x, R_y] = [L_y, R_x]$ together with the antisymmetry of the bracket gives us the equation  $[L_x,R_y]+[R_x,L_y]=0$, therefore $[S_x,S_y]=[L_x,L_y]+[R_x,R_y]$. An analogous argument shows that $[\operatorname{ad}_x,\operatorname{ad}_y]=[L_x,L_y]+[R_x,R_y]$ and, since $\operatorname{ad}$ is a representation, this entails that $[S_x, S_y] = \operatorname{ad}_{[x,y]} $.  
        \item[$3) \Rightarrow 1):$] The equality $[S_x, S_y] = \operatorname{ad}_{[x,y]} $ together with the fact that $\operatorname{ad}$ is a representation implies that  $[L_x,R_y]+[R_x,L_y]=0$. Evaluating this linear map at any point gives us precisely the AFA condition on the associator.  
    \end{enumerate}
\end{proof}
As a simple consequence, we have:

\begin{corollary}
If the bilinear product $\cdot: \g \times \g \mapsto \g, (x,y) \mapsto x\cdot y$ is an LSA structure and the right-insertion operator $R: \g \rightarrow \gl(\g)$ is an anti-representation of $\g$, then $\cdot$ is an AFA structure.      
\end{corollary}

Note, however, that for the bulk of what follows in this work, we will be less concerned with situations  where ${\cal A}$ is simultaneously a LSA, RSA and AFA. Rather, we work in the set-up where ${\cal A}$ is exclusively one of these algebras.

\subsection{Geometric Interpretation } 
 
To understand what AFAs geometrically refer to, it is instructive to examine the left and right insertion operators $L_x$ and $R_x$ respectively, which satisfy eq.~\eqref{csacon}.  Unlike LSAs or RSAs, for AFAs (that are not also LSAs or RSAs), $L$ or $R$ by themselves, are not  representations of the associated Lie algebra. However, $L - R$ yields a linear representation. In fact, this is the adjoint representation of the Lie algebra: 
\begin{eqnarray}
ad_x = L_x - R_x 
\label{ad}
\end{eqnarray}
By definition, elements $ad_x$ satisfy 
\[  [ ad_x , ad_y ] = ad_{ [x , y] }    \] 
Substituting eq.~(\ref{ad}) in the above identity and using eq.~(\ref{csacon}) gives
\begin{eqnarray}
[ L_x , L_y ]  - L_{ [x , y] }  =  -  [ R_x , R_y ]  - R_{ [x , y] } 
\label{csacon2}
\end{eqnarray}
In case of LSAs, we had that the left-hand side of eq.~(\ref{csacon2}) is identical to zero (with no relations between the $R_x$ operators); whereas for RSAs, the right-hand side of eq.~(\ref{csacon2}) becomes identical to zero (without any relations between the operators $L_x$). However, for AFAs (that are not also LSAs and RSAs), eq.~(\ref{csacon2}) presents a non-trivial "matching condition". To see what exactly is being matched here, recall that in the case of LSAs, the operator $L_x$ was identified with the affine connection $\nabla_x$ on a manifold $M$ with vector fields $x$. Then $[ L_x , L_y ]  - L_{ [x , y] }  = 0$ indicates the vanishing of the curvature tensor. Likewise, for RSAs, one can identity $- R_x$ with an affine connection $\widetilde{\nabla}_x$, and then $[ R_x , R_y ]  + R_{ [x , y] } = 0$ indicates a vanishing curvature with respect to the connection $\widetilde{\nabla}_x$.  

\begin{remark}
Therefore, for AFAs, eq.~(\ref{csacon2}) is geometrically interpreted as the admissibility of two connections $\nabla_x$ and $\widetilde{\nabla}_x$, each with non-vanishing curvature, such that the equality indicates a matching of the curvature tensors with respect to the two connections.  
\end{remark}

\subsection{Examples}

Let us now consider some examples of AFAs. Note that most of our discussion in this work will only involve finite-dimensional pre-Lie algebras, unless otherwise stated.   
\begin{example}
Consider a 3-dimensional algebra with generators $e_1, e_2, e_3$ defining the \\ following  product:  
\[ e_1 \cdot e_1 =  \, e_1  \qquad   e_2 \cdot e_2 = \lambda_2 \,\, e_2  \qquad   e_3 \cdot e_3 = \lambda_3 \,\, e_3     \]   
\[e_1 \cdot e_2 =  \, e_2  \qquad   e_1 \cdot e_3 = \,   e_3  \]  


and all other generator multiplications vanishing. The scalars $\lambda_2,  \lambda_3 \in {\mathbb C}$. At least, one of $  \lambda_2 $ or $ \lambda_3 $ must be non-zero.   Computing the associators, one confirms that this algebra is an AFA over ${\mathbb C}$. It is also straightforward to check that this algebra is not an LSA or RSA.  

The Lie algebra corresponding to the above AFA is:  

\[   [e_1 , e_2] =  \, e_2   \qquad    [e_1 , e_3] =  \, e_3
\]
and all other commutators being zero. Furthermore, one can check that this Lie algebra is solvable.

Besides the above-mentioned AFA, this Lie algebra is also known to admit a 3-dimensional LSA  (see Proposition 3.51 in \cite{burde2006left} for the case $\lambda = 1$). 


\end{example}

Note that a given Lie algebra may admit more than one pre-Lie structure. Hence, one may argue that pre-Lie structures serve as finer probes of geometry, rather than the corresponding Lie algebra (though not all Lie algebras may admit a pre-Lie structure). The full classification of all possible pre-Lie structures admissible by a given Lie algebra would be an important question to address in future work.  

\begin{example}
\label{ex2.2}
Consider another Lie algebra, one that is solvable: 
\[  [e_1 , e_2] =   e_2   \qquad    [e_3 , e_1] =  e_3  \]
This admits the following AFA:
\[ e_1 \cdot e_2 = \delta  \,\, e_2  \qquad   e_2 \cdot e_1 = - i \delta  \,\,   e_2      \]   
\[ e_3 \cdot e_1 = \delta  \,\,   e_3  \qquad  e_1 \cdot e_3 = - i \delta  \,\, e_3   \]  
where $\delta \in {\mathbb R}$.  This AFA leads to the stated Lie algebra after scaling the generators by $1/ ( \delta (1 + i) )$.

\end{example}

On the other hand, adding an extra relation to the above Lie algebra leads to
\[  [e_1 , e_2] =   e_2   \qquad    [e_3 , e_1] =  e_3 \qquad  [e_2 , e_3] = \frac{1}{2} e_1  
\]
which after scaling the three generators by $e_i \to 2 \, e_i$, gives  ${\mathfrak sl(2, \, {\mathbb C})}$. Do simple and semisimple Lie algebras such as  this admit AFAs or other pre-Lie structures? We will answer this question in Section 5.

\section{Lie-Admissibility and Solution Classes for AFAs}

We now turn our attention to Lie-admissibility criteria for AFAs with respect to a given Lie algebra. This analysis will help identify a few simple classes of solutions of AFAs. It will also serve as a computational procedure for determining the admissibility of an AFA, given the structure constants of a Lie algebra. With these methods, we will be able to find admissible AFAs for many solvable Lie algebras. Here, we will focus on projection order $p=1$ solutions only. A complete classification of all AFAs will be discussed in a future investigation. 

To set the notation, we consider any finite-dimensional  algebra over ${\mathbb C}$ with bilinear skew-symmetric brackets $[ \, , \, ]$, and generators $e_1, \; \cdots, \; e_n$. In this basis, the commutation relations take the form:
$ [e_i \, , \, e_j] = \sum_{k = 1}^n c_{ijk} \, e_k $  
where the structure constants $c_{ijk} \in {\mathbb C}$. 

The following definitions will be useful: 

\begin{definition}
\label{preal}
Consider an $n$-dimensional algebra ${\cal A}$ over a field ${\mathbb F}$, endowed with skew-symmetric brackets $[ \, , \, ]$ (though not necessarily a Lie algebra). $\left( {\cal A}, \, [ \, , \, ]  \, \right)$ has generators $\{ e_i \}$ $(1 \leq i \leq n)$, structure constants $c_{i j k}$, and defining relations $[e_i , e_j]  = \sum_{k = 1}^n \, c_{i j k}   \,\,  e_k$. We  define a  \textbf{Pre-Algebraic Structure} on ${\cal A}$ as the algebra $\left( {\cal A}, \; \cdot  \right)$ specified by the following relations:  
\begin{eqnarray}
 e_i \cdot e_j = \sum_{k = 1}^n \, d_{i j k}   \,\,  e_k 
\end{eqnarray}
such that 
\begin{eqnarray}
    [e_i \, , \, e_j] \, = \,  e_i \cdot e_j \, - \, e_j \cdot e_i
\end{eqnarray}
holds.

\end{definition}

From the above definition, it follows that the structure constants $d_{i j k}$ of $\left( {\cal A}, \; \cdot  \right)$ satisfy:  
\begin{eqnarray}
 d_{ijk} - d_{jik} = c_{ijk} 
\label{3.1.1} 
\end{eqnarray}
Note that pre-algebraic structures are not unique to a given algebra. Moreover, pre-Lie structures (in cases when they exist) are only a special subset of pre-algebraic structures. For this reason, pre-algebraic structures serve as a structural domain within which one may algorithmically look for possible pre-Lie structures (whenever those may be expected to exist) associated to a given Lie algebra.

Next, we introduce the notion of a `projection order' of an algebra as follows:
\begin{definition}
\label{porderdef}
Given a $n$-dimensional algebra $({\cal A}, \; \cdot)$ with product $e_i \cdot e_j  = \sum_{k = 1}^n d_{ijk} \; e_k$, in the basis $\{ e_i \}$, and structure constants $d_{ijk}$, the \textbf{Projection Order} $p$ is defined as the maximum of the $\ell_0$-norm of each product $e_i \cdot e_j$, over all pairs $i, \; j$
\begin{eqnarray}
     p = Max_{(i, \; j)} \; \{ \; | e_i \cdot e_j |_0  \;\; | \; 1 \leq i, j \leq n  \}  
\end{eqnarray}
where for any vector $v \in \mathbb{C}^n$, $|v|_0$ is defined as the number of non-zero components of $v$. The natural isomorphism between ${\cal A}$ and $\mathbb{C}^n$, where each $e_i$ maps to the vector in $\mathbb{C}^n$ with zeros everywhere except for $1$ at the $i$-th entry, equips ${\cal A}$ with an  $\ell_0$ norm. 
\end{definition}
For example, if $e_i \cdot e_j = d_{ijk} \; e_k$ (no summation over $k$ implied), then $p = 1$; on the other hand, for the same example, if there exists a pair $i, \; j$ such that $e_i \cdot e_j = d_{ijk_1}  e_{k_1} + d_{ijk_2}  e_{k_2}$, then $p = 2$. 

Note that the projection order as defined above is basis-dependent. Alternatively, one may define a canonical projection order, which minimizes $p$ over all possible bases choices. However, for our purposes here, Definition \ref{porderdef}  will suffice.  

In what follows, it will be useful to consider a class of algebras more general than Lie algebras, namely, Akivis algebras  \cite{akivis1976local}. These are defined as follows: 
\begin{definition}
\label{akivis}
An \textbf{Akivis Algebra} ${\cal A}$ over a field ${\mathbb F}$ is defined via an anti-symmetric bilinear product $[ x, y ]$ and a multi-linear ternary operation $\langle x, y, x \rangle$ such that the condition
\[ \sum_{\sigma \in S_{3}} \text{sgn}(\sigma) \langle x_{\sigma(1)}, x_{\sigma(2)}, x_{\sigma(3)} \rangle = J(x_{1}, x_{2}, x_{3})   \]
holds. Here, the Jacobinator is defined as 
\[ J(x_{1}, x_{2}, x_{3}) := \sum_{\sigma \in A_{3}} [[x_{\sigma(1)}, x_{\sigma(2)}], x_{\sigma(3)}] \]
and $A_3$ denotes the subgroup of $S_3$ corresponding to all three cyclic permutations of the set $\{ 1, 2, 3 \}$. 
\end{definition}
Note that, given any nonassociative algebra $({\cal A}, \; \cdot)$, a skew-symmetric product defined by $[ x, y ] := x \cdot y \, - \, y \cdot x$, along with a ternary operation $\langle x,y,z \rangle := (x \cdot y) \cdot z \, - \, x \cdot (y \cdot z)$, defines an Akivis algebra $\left( {\cal A}, \, [ \, , \, ]  \, \right)$ on ${\cal A}$. The pre-algebraic structures $({\cal A}, \; \cdot)$ that we will work with, in general, define an Akivis algebra. When the AFA admissibility conditions are satisfied, the latter yields a Lie algebra.

\subsection{Determining Lie-Admissibility }

In order to realize an AFA, a pre-algebraic structure (eq.~\eqref{3.1.1}) on  $\left( {\cal A}, \, [ \, , \, ]  \, \right)$ ought to satisfy the following: 
\begin{eqnarray}
 (e_i , e_j , e_j) \, &=& \, (e_j , e_j , e_i)   \label{3.1.2}  \\
 (e_i , e_j , e_m) \, &=& \, (e_m , e_j , e_i)    \label{3.1.3} 
\end{eqnarray}
for distinct $i, j, m$ from $1$ to $n$, and yield non-trivial solutions for $d_{ijk}$. 

For an algebra of projection order $p$, we have at most $\frac{n (n - 1) p}{2}$ structure constants $c_{ijk}$; and an equal number (to that of the $c_{ijk}'s$) of independent parameters $d_{ijk}$ in its pre-algebraic structure. The remaining parameters $d_{jik}$ are all constrained by eq.~\eqref{3.1.1}.


For any given projection order $p$, eqs. \eqref{3.1.2} and \eqref{3.1.3} then yield:  
\begin{eqnarray}
\sum_{a = 1}^p \sum_{r=1}^p \left( d_{ijk_r} \; d_{k_r j l^r_a} \; e_{l^r_a} \;   -    d_{i \overline{k}_r \overline{l^r}_a } \; d_{j j \overline{k}_r} \; e_{\overline{l^r}_a} \right)  =  \sum_{a = 1}^p \sum_{r=1}^p \left( - d_{j k_r l^r_a} \; d_{j i k_r} \; e_{l^r_a} \;   +    d_{j j \overline{k}_r} \; d_{\overline{k}_r i \overline{l^r}_a} \; e_{\overline{l^r}_a}   \right)   
\label{3.1.4}
\end{eqnarray}
and 
\begin{eqnarray}
\sum_{a = 1}^p \sum_{r=1}^p \left( d_{ijk_r} \;  d_{k_r m l^r_a} \; e_{l^r_a} \;   -    d_{i \overline{k}_r \overline{l^r}_a } \; d_{j m \overline{k}_r} \; e_{\overline{l^r}_a}  \right)  =  \sum_{a = 1}^p \sum_{r=1}^p  \left( - d_{m k_r l^r_a} \; d_{j i k_r} \; e_{l^r_a} \;  +    d_{m j \overline{k}_r} \; d_{\overline{k}_r i \overline{l^r}_a} \; e_{\overline{l^r}_a}  \right) 
 \label{3.1.5} 
\end{eqnarray}
respectively. The generators $e_{l^r_a}$, $e_{\overline{l^r}_a}$ have been included above to distinguish terms that are linearly independent. 
From a counting of the number of conditions imposed by eqs.~\eqref{3.1.4} and \eqref{3.1.5}, one can explicitly determine whether this system of equations has a non-trivial solution, or is overdetermined.

\subsubsection*{The $p = 1$ Case }

We will consider the case of projection order $p = 1$ in  detail. Here, the sums over $r$ and $a$ drop in eqs.~\eqref{3.1.4} and \eqref{3.1.5}; and for given $i$ and $j$ in $d_{ijk}$, the index $k$ takes a fixed value from $1$ to $n$. 

We will now derive four classes of possible solutions to the above equations based on the index structure of the parameters $d_{ijk}$. This yields four kinds of AFAs.  To study each class, we will consider representative indices $i, j, m$ mutually distinct. Then we have the following possibilities: \\ 
(I)  when $k$ is different from $i$ and $j$ in $d_{ijk}$,  \\ 
(II) when $k = i$, \\   
(III) when only associators of the type $(e_i , e_j , e_i)$ or / and $(e_i , e_i , e_i)$ are non-vanishing, \\
(IV) when $k = j$.

{\bf Solution Class I}: 

Let us start with class I. This involves terms in eq.~\eqref{3.1.4} that depend on generators $e_l$. Substituting eq.~\eqref{3.1.1} into these terms gives 
\begin{eqnarray}
  d_{ijk} \; ( d_{jkl}  -  c_{jkl} )  \; = \,  - \; d_{jkl} \; ( d_{ijk}  -  c_{ijk} )    \label{3.1.6}  
\end{eqnarray}
Taking into account each pair $i, j$, there are generally $n (n - 1)$ such equations in $\frac{n (n - 1)}{2}$ variables. Such a system of equations is overdetermined whenever the number of independent equations is greater than the number of non-vanishing independent variables. Hence, the only solution satisfying the above equations is: 
\[   d_{ijk} \; = \; c_{ijk}   \]
which implies
\[ d_{jik} = 0 \]
corresponding to each $d_{ijk}$. That implies that every associator in $\left( {\cal A}, \; \cdot \;  \right)$ vanishes. That is,  $\left( {\cal A}, \; \cdot \;  \right)$ is not a non-trivial AFA. The additional conditions of this class from eq.~\eqref{3.1.5} only further reinforce this point. 

{\bf Solution Class II}: 

Now, let us look at class II. Substituting eq.~\eqref{3.1.1} into eq.~\eqref{3.1.4} (for terms involving the generators $e_l$) with $k = i$ and $p = 1$ gives  
\begin{eqnarray}
  d_{iji} \;  d_{iji}   \; = \,  -  \; ( d_{iji}  -  c_{iji} ) \; ( d_{iji}  -  c_{iji} )     \label{3.1.61}
\end{eqnarray}
Solving for $d_{iji}$ yields 
\begin{eqnarray}
   d_{iji} = \frac{1 \pm i }{2} \; c_{iji}    \label{3.1.7}  
\end{eqnarray}
Using eq.~\eqref{3.1.1} in eq.~\eqref{3.1.7} we get 
\begin{eqnarray}
 d_{jii} \, =  \, \pm \, i \; \delta \; c_{iji} \, = \, \pm \, i \; d_{iji}
 \label{3.1.7III}
\end{eqnarray}
where $\delta$ denotes $( 1 \pm i )/2$.  
Solutions of this type realize valid AFAs as long as the system of equations is not overdetermined. The above are $\frac{n (n - 1)}{2}$ equations in $\frac{n (n - 1)}{2}$ variables. There is also another set of $\frac{n (n - 1)}{2}$ equations coming from \eqref{3.1.4} when indices $i \leftrightarrow j$ are swapped, but $k = i$ is held fixed. Here, parameters of the form $d_{iil}$ have to vanish to ensure that these extra constraints do not overdetermine the system. Then, the $n (n - 1)$ equations in eq.~\eqref{3.1.4} reduce to $\frac{n (n - 1)}{2}$ equation with as many variables, leading to the solution mentioned above. 

However, we are not done yet. Notice that eq.~\eqref{3.1.5} introduces additional constraints to the above parameters. Except when those additional equations vanish identically, the system becomes overdetermined. A quick check of the equations and variables demonstrates this:      
Eq.~\eqref{3.1.5} consists of $2 n (n - 1) (n - 2)$ conditions since each of the $i, j, m$ are independent indices, and the factor $2$ follows from the fact that for $p = 1$ the two terms on the LHS (respectively, RHS) of eq.~\eqref{3.1.5} multiply distinct generators.  Taking $k = i$ (for fixed $i, j, m$) in eq.~\eqref{3.1.5} gives the following conditions: 
\begin{eqnarray}
d_{iji} \;  d_{iml}   \; = \,  -  \;  d_{mil} \;  d_{jii}   \label{3.1.8}   \\
 - \; d_{ipq} \;  d_{jmp}   \; = \,   d_{mjp} \;  d_{piq}  
\end{eqnarray}
Then, writing eq.~\eqref{3.1.5} with indices $i \leftrightarrow j$ swapped and fixed $k = i$ gives 
\begin{eqnarray}
d_{jii} \;  d_{iml}  \;   = \,  -  \;  d_{mil} \;  d_{iji}   \label{3.1.9}  \\  
- \; d_{jlr} \;  d_{iml}   \; = \,   d_{mil} \;  d_{ljr} 
\end{eqnarray}
Notice that eq.~\eqref{3.1.8} is satisfied when $d_{mil} = d_{jii}$, while eq.~\eqref{3.1.9} requires $d_{mil} = d_{iji}$, which leads to a contradiction since $d_{iji} \neq d_{jii}$. Thus, indicating that the system is overdetermined unless sufficiently many of the conditions in eq.~\eqref{3.1.5} identically vanish. The remaining equations with parameters that are not of the form $d_{iji}$ or $d_{jii}$ are still covered by solutions corresponding to class I. 

{\bf Solution Class III}: 

Next, we have class III. Solutions of this class refer to the case when all associators of the algebra, except those of the form $(e_i , e_j , e_i)$ or / and $(e_i , e_i , e_i)$, vanish. Let us first consider the former instance. When this happens, we have: 
\begin{eqnarray}
  (e_i , e_j , e_i) \, = \,  d_{iil} \; c_{iji} \; \, e_l
\end{eqnarray}
This associator involves parameters of the form $d_{iji}$ (for parameters with  other index structures, this associator vanishes identically). Hence, the corresponding AFA only requires non-vanishing $d_{iil}$ and $d_{iji}$. The existence of the former is guaranteed from terms in eq.~\eqref{3.1.5} coming from the generators $e_{\overline{l}}$ (when ${\overline{k}} = i$). And $d_{iji} = c_{iji}$ arises from terms in eq.~\eqref{3.1.5} coming from the generators $e_l$ (when ${k} = i$). 

Analogously, when
\begin{eqnarray}
  (e_i , e_i , e_i) \, = \,  d_{ iik } \; c_{ kil  } \; \, e_l
\end{eqnarray}
is non-vanishing, we have non-trivial $d_{iik}$ and $d_{kil}$ coming from terms in eq.~\eqref{3.1.5}. In the examples that will follow, we shall demonstrate a Lie algebra arising from this class of AFAs.   

{\bf Solution Class IV}: 

Finally, let us examine class IV, where we take $k = j$. The terms in eq.~\eqref{3.1.4} involving generators $e_l$ now become
\begin{eqnarray}
  d_{ijj} \;  d_{jjl}    \; = \,  - \; d_{jij} \;  d_{jjl}      \label{cII1}  
\end{eqnarray}
when $d_{jjl} \neq 0$ holds, then we have
\begin{eqnarray}
  d_{ijj}   \; = \, \frac{c_{ijj}}{2} \, = \, - \; d_{jij}        
\label{cII2}  
\end{eqnarray}

Moreover, the terms in eq.~\eqref{3.1.4} involving generators $e_{\overline{l}}$   also fall into this class (though the indices $i, j$ are organized differently for these terms). Consider 
\begin{eqnarray} 
- \, d_{i \overline{k} \overline{l} } \; d_{j j \overline{k}} \, = \, d_{j j \overline{k}} \; d_{\overline{k} i \overline{l}} 
\label{cII3}  
\end{eqnarray}
When $d_{j j \overline{k}} \neq 0$ holds, then we have
\begin{eqnarray}
  d_{i \overline{k} \overline{l} }   \; = \, \frac{c_{i \overline{k} \overline{l} }}{2} \, = \, - \; d_{\overline{k} i \overline{l}}       \label{cII4}  
\end{eqnarray}

Likewise, solutions of the above form (eq.~\eqref{cII2} and eq.~\eqref{cII4}) also show up from conditions imposed in eq.~\eqref{3.1.5}. However, note that specific instances of the above solutions yield a valid AFA only when they are consistent with all the other conditions imposed by eqs.~\eqref{3.1.4} and \eqref{3.1.5}. As we shall see in the examples, solutions of this kind are closely related to conditions in class III (class IV solutions can also be obtained combining conditions of class III).

\subsubsection*{The $p > 1$ Case }

Let us also briefly examine what happens when $p > 1$. Rather than running through each solution class, here we will focus mostly on parameter counting to determine whether or not the system of equations is overdetermined.  

Firstly, let us look at the $p = 2$ reduction of eqs. \eqref{3.1.4} and \eqref{3.1.5}. The terms in eq.~\eqref{3.1.4} associated to generators $e_{l^r_a}$ yield: 
\begin{eqnarray}
     \sum_{a = 1}^2  d_{ijk_1} \; d_{k_1 j l^1_a} \; e_{l^1_a} \; + \;   \sum_{b = 1}^2  d_{ijk_2} \; d_{k_2 j l^2_b} \; e_{l^2_b} \; = \; -   \sum_{a = 1}^2  d_{j k_1 l^1_a} \; d_{j i k_1} \; e_{l^1_a} \; -  \sum_{b = 1}^2  d_{j k_2 l^2_b} \; d_{j i k_2} \; e_{l^2_b}
 \label{3.1.10}
\end{eqnarray}
This equation leads to a maximum of $4n(n-1)$ conditions on the parameters corresponding to the generators $e_{l^1_a}$ and $e_{l^2_b}$; and $2n(n-1)$ conditions when $e_{l^1_a} = e_{l^2_b}$ is the case. The parameters $d_{ijk_1}$ and $d_{ijk_2}$ correspond to the structure constants $c_{ijk_1}$ and $c_{ijk_2}$, respectively. Hence, there are at most $2n(n - 1)/2$ independent parameters of type $d_{ijk}$.  

The remaining terms of eq.~\eqref{3.1.4} associated to generators $e_{\overline{l^r}_a}$ give: 
\begin{eqnarray}
    -   \sum_{a = 1}^2  d_{i \overline{k}_1 \overline{l^1}_a } \; d_{j j \overline{k}_1} \; e_{\overline{l^1}_a} \; -  \sum_{b = 1}^2  d_{i \overline{k}_2 \overline{l^2}_b} \; d_{j j \overline{k}_2} \; e_{\overline{l^2}_b} =  \sum_{a = 1}^2  d_{j j \overline{k}_1} \; d_{\overline{k}_1 i \overline{l^1}_a} \; e_{\overline{l^1}_a} \; + \;   \sum_{b = 1}^2  d_{j j \overline{k}_2} \; d_{\overline{k}_2 i \overline{l^2}_b} \; e_{\overline{l^2}_b} 
\label{3.1.10part2}
\end{eqnarray}
The counting is similar to that in eq.~\eqref{3.1.10} above with a maximum of $4n(n-1)$ conditions. There are at most $2n(n - 1)/2$ independent parameters of type $d_{i\overline{k} \overline{l}}$, and at most $2n$ independent parameters of type       $d_{j j \overline{k} }$. Of course, the conditions in eqs.~\eqref{3.1.10} and \eqref{3.1.10part2} are not independent of each other as they involve overlapping parameters.

Examining eq.~\eqref{3.1.5}, we have:  
\begin{eqnarray}
&& \sum_{a = 1}^2  d_{ijk_1} \; d_{k_1 m l^1_a} \; e_{l^1_a} \; + \;   \sum_{b = 1}^2  d_{ijk_2} \; d_{k_2 m l^2_b} \; e_{l^2_b} \;  - \;   \sum_{c = 1}^2  d_{i k_3 l^3_c} \; d_{j m k_3} \; e_{l^3_c} \; - \;  \sum_{d = 1}^2  d_{i k_4 l^4_d} \; d_{j m k_4} \; e_{l^4_d}   \nonumber  \\  
&=&  - \;  \sum_{a = 1}^2  d_{m k_1 l^1_a} \; d_{j i k_1} \; e_{l^1_a} \; - \;  \sum_{b = 1}^2  d_{m k_2 l^2_b} \; d_{j i k_2} \; e_{l^2_b}  \; + \;  \sum_{c = 1}^2  d_{m j k_3} \; d_{k_3 i l^3_c} \; e_{l^3_c} \; + \;   \sum_{d = 1}^2  d_{m j k_4} \; d_{k_4 i l^4_d} \; e_{l^4_d}   \qquad 
 \label{3.1.11}
\end{eqnarray}
which imposes a maximum of $8n(n-1)(n-2)$ additional conditions on the parameters $d_{ijk}$, and $2n(n-1)(n-2)$ when generators $e_{l_a} = e_{l_b} = e_{l_c} = e_{l_d}$. Analogous to the $p=1$ set-up, similar solution classes also  appear for $p=2$, but now there is the possibility of mixing between  the classes, and also the possibility of completely new solution classes beyond the ones we  had for $p=1$. The maximum possible number of parameters $d_{ijk}$ is $n(n-1)$. For a given algebra, when all of these parameters are non-vanishing, then the number of conditions imposed on the $n(n-1)$   $d_{ijk}$s  by eq.~\eqref{3.1.10} is $2n(n-1)$, by eq.~\eqref{3.1.10part2} is another $2n(n-1)$, and an additional $2n(n-1)(n-2)$ conditions come from eq.~\eqref{3.1.11}. The same counting arguments as in $p=1$ apply here, and this is true even when one considers mixtures of the four solution classes. 

In general, for any $p$, we have at most $pn(n-1)/2$ independent parameters $d_{ijk}$ that one has to solve for. Eq.~\eqref{3.1.4} imposes a maximum of $p^2 n(n-1)$ conditions on the  $d_{ijk}$s. 
In addition, eq.~\eqref{3.1.5} imposes another set of $2 p^2 n(n-1)(n-2)$ conditions, at most. The actual number of parameters and equations that one has to solve depends on the specific algebra at hand. When the system is overdetermined, there is no admissible AFA. Otherwise, a non-trivial AFA can be found. In the examples that follow, we demonstrate some of these computations.  


\subsection{Computing AFAs }

Here, we consider examples of solvable Lie algebras for which the above admissibility criteria yield realizations of AFAs (restricting to examples with $p=1$).

\begin{example}
Let us begin with the same solvable Lie algebra that we encountered earlier in example \ref{ex2.2}: 
\[  [e_1 , e_2] =   e_2   \qquad    [e_3 , e_1] =  e_3  \]
For this, eqs.~\eqref{3.1.4} and \eqref{3.1.5} yield 
\begin{eqnarray}
d(2,1,2) \; d(2,1,2) &=& - \, d(1,2,2) \; d(1,2,2) \nonumber \\  
d(3,1,3) \; d(3,1,3) &=& - \, d(1,3,3) \; d(1,3,3)  \nonumber      
\end{eqnarray}
We get two decoupled equations, each involving one independent variable. These equations correspond to class II (with $p = 1$). The solution is thus:    
\begin{eqnarray*}
    d(2,1,2) =  \delta  \qquad \qquad    \; d(1,2,2)  = \pm \, i \, \delta  \\  
   d(3,1,3) =  \delta   \qquad  \qquad   \; d(1,3,3)  = \pm \, i \, \delta   
\end{eqnarray*}
where $\delta = ( 1 \pm i )/2$. Hence, this solvable Lie algebra admits a non-trivial AFA of class II. 
 
\end{example}

\begin{example}
On the other hand, the solvable Lie algebra with generators $e_1, e_2, e_3$ 
\[  [e_1 , e_2] =   e_1  \]
admits an AFA of class III. The symmetric non-vanishing associator can be read-off from the Lie algebra as
\begin{eqnarray*}
  (e_1 , e_2 , e_1) \, = \,  d(1,1,1) \, c_{121} \, e_1     
\end{eqnarray*}
Matching the terms in eq.~\eqref{3.1.5} corresponding to the $e_l$ and $e_{\overline{l}}$ generators, respectively, gives: 
\begin{eqnarray*}
 d(1,2,1) \; d(1,m,l) &=& - \, d(m,1,l) \; d(2,1,1)    \\
- d(1,1,1) \; d(2,m,1) &=&  \, d(m,2,1) \; d(1,1,1) 
\end{eqnarray*}  
Since $m \neq 1,2$ here, we get the following solution  
\begin{eqnarray*}
    d(1,1,1) =  1  \qquad \qquad  \; d(2,2,2)  =  1   \qquad \qquad  \;   d(1,2,1) = 1    
\end{eqnarray*}
with all other parameters being zero. Notice that there is an additional non-zero parameter $d(2,2,2)  =  1$. This comes from requiring that the associator $(e_1 , e_2 , e_2)$ vanish.  
    
\end{example}

\begin{example}
Finally, let us reconsider the same Lie algebra as in the previous example but find a class IV type solution for it. 
\[  [e_1 , e_2] =   e_1  \]
As before 
\begin{eqnarray*}
  (e_1 , e_2 , e_1) \, = \,  d(1,1,1) \, c_{121} \, e_1     
\end{eqnarray*}
But now we should have a non-vanishing $d(2,1,1)$ to be a class IV solution. This means there is at least one more non-trivial associator, which happens to be 
\begin{eqnarray*}
  (e_2 , e_2 , e_2) \, = \,  d(2,2,1) \, c_{121} \, e_1   
\end{eqnarray*}
which implies that we have a non-vanishing $d(2,2,1)$. 
Once again, matching the terms in eq.~\eqref{3.1.5} corresponding to the $e_l$ and $e_{\overline{l}}$ generators, respectively, gives: 
\begin{eqnarray*}
 d(1,2,1) \; d(1,m,l) &=& - \, d(m,1,l) \; d(2,1,1)    \\
- d(1,1,1) \; d(2,m,1) &=&  \, d(m,2,1) \; d(1,1,1) 
\end{eqnarray*}  
And swapping indices $i \leftrightarrow j$ in eq.~\eqref{3.1.5} gives:   
\begin{eqnarray*}
 d(2,1,1) \; d(1,m,l) &=& - \, d(m,1,l) \; d(1,2,1)    \\
- d(2,2,1) \; d(1,m,2) &=&  \, d(m,1,2) \; d(2,2,1) 
\end{eqnarray*} 
which yields the following solution  
\begin{eqnarray*}
    d(1,1,1) =  1  \qquad \qquad  \; d(2,2,1)  =  1   \qquad \qquad  \;   d(1,2,1) = \frac{1}{2}  \qquad \qquad  \;  d(1,2,1) = - \, \frac{1}{2}    
\end{eqnarray*}

\end{example}

\section{$A_3$-Associative and $S_3$-Associative Algebras} 

Besides LSAs, RSAs and AFAs, there also exist two other classes of nonassociative Lie-admissible algebras \cite{remm2002operades, goze2004lie}. We state these definitions below and provide examples.  

\begin{definition}
One defines an {\bf ${\mathbf{A_3}}$-Associative Algebra} as an algebra $\left( {\cal A}, \cdot  \right)$ over a field ${\mathbb F}$ equipped with a bilinear product  $(x , y) \mapsto x \cdot y$ such that the associator satisfies 
\begin{eqnarray}
    (x , y , z) + (y , z , x) + (z , x , y) = 0 
\label{caadef}
\end{eqnarray}
 for all $x, y, z \in {\cal A}$. 
\end{definition}
Using eq.~\eqref{caadef} in eq.~\eqref{eq:jacobi_identity}, one easily verifies that the commutator algebra based on $\left( {\cal A}, \cdot  \right)$ satisfies the Jacobi identity. As noted in \cite{goze2004lie}, eq.~\eqref{caadef} is simply the signed sum of associators in the subgroup of $S_3$ that contains even permutations; that is, the alternating group $A_3$. Hence, the name: $A_3$-associative algebra.

\begin{example}
An example of an $A_3$-associative algebra is the following:  
\[ e_1 \cdot e_1 =  \, e_1  \qquad   e_2 \cdot e_2 = \, e_2  \qquad   e_3 \cdot e_3 = \, e_3     \]   
\[e_1 \cdot e_2 =  \, e_3  \qquad   e_2 \cdot e_3 = \,   e_1  \qquad   e_3 \cdot e_1 = \,   e_2  \]  
\[e_2 \cdot e_1 =  \, e_3  \qquad   e_3 \cdot e_2 = \,   e_1  \qquad   e_1 \cdot e_3 = \,   e_2  \] 

This happens to be a three-dimensional commutative and nonassociative algebra that satisfies the axioms of a dihedral quandle of order 3. Quandles are algebraic structures studied in knot theory, that encapsulate the Reidemeister moves \cite{carter2012survey}. Perhaps what may not have been noted in the knot theory literature is that this particular involutory quandle also happens to be a Lie-admissible algebra.  
 
\end{example}
However, notice that the commutator algebra based on the example above leads to an abelian Lie algebra. In the next section, we will provide another example of an $A_3$-associative algebra that serves as a pre-Lie structure of an interesting non-abelian Lie algebra. 

Furthermore, $A_3$-associative algebras have an interesting geometric interpretation. Identifying the connection $\nabla_x$ on a manifold $M$ with the left insertion operator $L_x$, the curvature tensor can be expressed as 
\[ {\cal R}_L (x,y) \ z \, := \,  \left( \, [ L_x , L_y ]  - L_{ [x , y] } \, \right) \, z  \]
Then, the $A_3$-associativity condition gives 
\[ {\cal R}_L (x,y) \ z \, + \, {\cal R}_L (y, z) \ x \, + \, {\cal R}_L (z, x) \ y  \, = \, 0  \]
which is precisely the first Bianchi identity. An analogous identity is obtained using the right insertion operator $R_x$, with the associated curvature tensor ${\cal R}_R (x,y) \ z$.

Next, we have the following definition:
\begin{definition}
A \textbf{$\mathbf{S_3}$-Associative Algebra} as an algebra $\left( {\cal A}, \cdot  \right)$ over a field ${\mathbb F}$ equipped with a bilinear product  $(x , y) \mapsto x \cdot y$ such that the associator satisfies 
\begin{eqnarray}
    (x , y , z) + (y , z , x) + (z , x , y) = (y , x , z) + (x , z , y) + (z , y , x)
\label{daadef}
\end{eqnarray} 
 for all $x, y, z \in {\cal A}$. 
\end{definition}

In the terminology of \cite{goze2004lie}, this Lie-admissible algebra is the signed sum of associators, over $S_3$ as its own subgroup. Hence, the name: $S_3$-associative algebra.

\begin{example}
\label{eg2.4S3}
An example of a $S_3$-associative algebra is the following: 
\[ e_1 \cdot e_1 =  \, e_1  \qquad   e_2 \cdot e_2 = \, e_2  \qquad   e_3 \cdot e_3 = \, e_3     \]   
\[e_1 \cdot e_2 = \alpha \, e_3  \qquad   e_2 \cdot e_3 = \alpha  \, e_1  \qquad   e_3 \cdot e_1 = \alpha \,   e_2  \]  
\[e_2 \cdot e_1 = \beta \, e_3  \qquad   e_3 \cdot e_2 = \beta \,   e_1  \qquad   e_1 \cdot e_3 = \beta \,  e_2  \] 

where $\alpha \neq \beta$ are real constants (the case $\alpha = \beta$ simply yields an $A_3$-associative algebra).
    
\end{example}

Notice that $S_3$-associative algebras subsume every other class of Lie-admissible algebras, LSAs, RSAs, AFAs, $A_3$-associative algebras, as well as ordinary associative algebras. This is because this class admits the action of the full $S_3$ group on the set of associators. Of course, there also exist $S_3$-associative algebras which are exclusively $S_3$-associative, such as the one in the example above.

\section{Semisimple Lie Algebras and Pre-Lie Structures}

\subsection{Semisimple Lie Algebras and Admissibility of AFAs}

It is well-known that semisimple Lie algebras do not admit LSA structures (respectively, RSA). Thus, a natural question is: does the same thing hold for AFAs? In other words, is it true that if $\mathfrak{g}$ is a semisimple Lie algebra, then $\mathfrak{g}$ cannot be endowed with an AFA structure? In general, this turns out to be false.

We construct the following counterexample: 
\begin{example}
Consider the simple Lie algebra $\g=\mathfrak{sl}(2,\mathbb{C})$, whose Lie bracket is given by:  
\[  [ H , E] =  2 \, E   \qquad    [H , F] = - 2 \, F \qquad  [E , F] =  H  
\]
It is well-known that $\mathfrak{g}$ admits a grading given by
$$
\g=\g_{-1} \oplus \g_0 \oplus \g_1
$$
where $\g_{-1}=\langle F \rangle$, $\g_0=\langle H \rangle$ and $\g_1=\langle E \rangle$. As the above decomposition gives us a grading, we have that $[\g_i,\g_j] \subset \g_{_{i+j}}$ for all $i,j$, where $\g_i:=0$ for $|i| \geq 2$. Therefore, in order to construct an AFA product that is compatible with the grading, we may impose the condition $\g_i \g_j \subset \g_{i+j}$. This last condition imposes the following relations: 
$$
F \cdot F=0, \hspace{5mm} E \cdot E=0.
$$
Next, by Proposition \ref{prop: commutative_symmetric_algebra} we know that an AFA product cannot be commutative or skew-symmetric. Therefore, in order to get a product that is as asymmetric as possible and in such a way that the AFA condition holds, we also impose the following conditions: 
\begin{equation*}
    \g_1 \,  \g_0=0, \hspace{5mm} \g_{-1} \, \g_1=0 \; \textrm{ and } \; \g_0 \,  \g_{-1}=0.
\end{equation*}
This gives us the additional relations: 
$$
E \cdot H:=0, \hspace{5mm} F \cdot E:=0, \hspace{5mm} H \cdot F:=0.
$$
We can derive the other relations by using the equation $[x,y]=x \cdot y - y \cdot x$ together with the AFA equation. First, the former equation implies that:
\begin{equation*}
    \begin{split}
        H & =[E,F]=E \cdot F-F \cdot E=E \cdot F   \Rightarrow E \cdot F=H  \\
        -2F&=[H,F]=H \cdot F-F \cdot H=-F \cdot H  \Rightarrow F \cdot H=2F \\
        2E&=[H,E]=H \cdot E-E \cdot H=H \cdot E \Rightarrow H \cdot E=2E
    \end{split}
\end{equation*}
Finally, as $\g_0 \, \g_0 \subset \g_0$, there must be $\lambda \in \mathbb{C}$ such that $H \cdot H=\lambda H$. To determine $\lambda$ we now use the AFA condition $(x,y,z)=(z,y,x)$. On one hand, we should have:
$$
(F,E,H)=(F\cdot E) \cdot H-F \cdot (E \cdot H)=0-0=0.
$$
On the other hand, 
$$
(H,E,F)=(H\cdot E) \cdot F-H\cdot(E \cdot F)=2E \cdot F-H \cdot H=2H-\lambda H=(2-\lambda)H.
$$
The AFA condition implies $\lambda=2$. A straightforward verification shows that this product defines an AFA structure on $\g$. 
\label{sl2cex}
\end{example}

\begin{example}
Related to the above example, the standard change of basis from $\mathfrak{sl}(2,\mathbb{C})$ to $\mathfrak{su}(2)$ yields the Lie algebraic relations:
\[ [ \, X_i \, , \, X_j \, ] \, = \, 2 \,  \epsilon_{ijk} \, X_k  \]
Hence, $\mathfrak{su}(2)$ also admits an AFA, which is obtained by a basis transformation of the AFA in example \ref{sl2cex}: 
\[ X_1 \cdot X_1 =  \, i \, X_3  \qquad   X_2 \cdot X_2 = \, i \, X_3  \qquad   X_3 \cdot X_3 = \, 2 \, i \, X_3     \]   
\[ \; X_1 \cdot X_2 =  \, X_3  \qquad   X_2 \cdot X_3 = \,   X_1 + \, i \, X_2  \qquad   X_3 \cdot X_1 = \,  i \, X_1 +  X_2  \]  
\[ \;\; X_2 \cdot X_1 =  \, - X_3  \qquad   X_3 \cdot X_2 = \,  -  X_1 + \, i \, X_2   \qquad   X_1 \cdot X_3 = \,   i \, X_1 - X_2  \] 
Notice that even though this AFA is isomorphic to the one mentioned above, it does not follow from a Lie algebraic grading (which was crucial in that case), and it belongs to a different projection order ($p = 2$). 

\end{example}

\subsection{Semisimple Lie Algebras and $A_3$ / $S_3$ Associativity}

Besides AFAs, do semisimple Lie algebras admit other pre-Lie structures? Below, we shall answer this question in the affirmative. 

First, let us consider the following examples: 
\begin{example}  
The standard cross-product algebra in 3 dimensions is:  
\[ X_1 \cdot X_1 =  \, 0  \qquad   X_2 \cdot X_2 = \, 0  \qquad   X_3 \cdot X_3 = \, 0     \]   
\[ \; X_1 \cdot X_2 =  \, X_3  \qquad   X_2 \cdot X_3 = \,   X_1  \qquad   X_3 \cdot X_1 = \,   X_2  \]  
\[ \;\; X_2 \cdot X_1 =  \, - X_3  \qquad   X_3 \cdot X_2 = \,  - X_1  \qquad   X_1 \cdot X_3 = \,  - X_2  \] 

Using the Jacobi identity of the cross-product algebra, it follows that
\[ \sum_{\sigma \in A_{3}}  ( X_{\sigma(1)}, X_{\sigma(2)}, X_{\sigma(3)} ) = 0 \]
which yields an $A_3$-associative algebra.

As is well-known, the commutator algebra of the cross-product in 3 dimensions is $\mathfrak{su}(2)$. 

\end{example}

\begin{example}  
Another example is the $S_3$-associative algebra we had in Example \ref{eg2.4S3} above. When $\alpha \neq \beta$, this is exclusively $S_3$-associative. Its commutator algebra gives $[ \, e_i \, , \, e_j \, ] \, = \, (\alpha - \beta) \,  \epsilon_{ijk} \, e_k$, which, upon rescaling generators, is simply  $\mathfrak{su}(2)$. Hence, $\mathfrak{su}(2)$ admits both, $A_3$-associative and $S_3$-associative pre-Lie structures (in addition to the AFA seen earlier). 
\end{example}

In fact, it is straightforward to prove that:  
\begin{theorem}
$S_3$-associative algebras realize  universal pre-Lie structures for all Lie algebras over $\mathbb C$.  \end{theorem}
\begin{proof}
Given a $n-$dimensional Lie algebra $({\cal A}, \, [ \, , \,] \,)$, any pre-algebraic structure (assuming its existence) on ${\cal A}$ given by $({\cal A}, \, \cdot \,)$, and  satisfying Definition \ref{preal},  satisfies  eq.~\eqref{eq:jacobi_identity} (also known as the Akivis identity), with the associator defined by  $( x,y,z ) := (x \cdot y) \cdot z \, - \, x \cdot (y \cdot z)$.   Hence, we have $\sum_{\sigma \in S_{3}} \text{sgn}(\sigma) ( x_{\sigma(1)}, x_{\sigma(2)}, x_{\sigma(3)} ) = J(x_{1}, x_{2}, x_{3})$. But for a Lie algebra, the Jacobinator $J(x_{1}, x_{2}, x_{3}) = 0$, which implies     $\sum_{\sigma \in S_{3}} \text{sgn}(\sigma) ( x_{\sigma(1)}, x_{\sigma(2)}, x_{\sigma(3)} )  = 0$, and that, in turn, satisfies the definition of a $S_3$-associative algebra over ${\cal A}$.  
\end{proof}

Hence, semisimple Lie algebras over $\mathbb C$ admit multiple pre-Lie structures: from associative algebras, to AFAs, to $A_3$-associative, to $S_3$-associative algebras. The only ones they exclude are LSAs and RSAs.

\section{Conclusions and Discussion }

In this work, we have addressed the question of admissibility of pre-Lie structures for semisimple Lie algebras over ${\mathbb C}$. Among the five classes of nonassociative Lie-admissible algebras, LSAs and RSAs were already known to be non-admissible by semisimple Lie algebras of finite dimension $n \geq 3$   \cite{helmstetter1979radical, burde1994left, burde1998simple, burde2006left}. The third class: AFAs, were not extensively discussed before in this context. Contrary to what one may expect, here we present an example demonstrating that semisimple Lie algebras do not exclude AFAs as pre-Lie structures.  We then discuss the remaining two classes of nonassociative Lie-admissible algebras, the $A_3$-associative and $S_3$-associative types, and prove that  $S_3$-associative algebras are universal pre-Lie structures for any Lie algebra over ${\mathbb C}$, including semisimple ones. If indeed the only pre-Lie structures that semisimple Lie algebras exclude (for dimensions $n \geq 3$) are the ones associated to LSAs and RSAs, that implies the manifolds associated to semisimple Lie groups only admit torsion-free connections that are not flat. 

Our findings here also suggest an intriguing geometric interpretation: Recall that the six Lie-admissible algebras are uniquely labeled by the elements of the permutation group $S_3$. LSAs, RSAs and AFAs refer to the same conjugacy class of $S_3$, namely, those of order 2. The $A_3$- and $S_3$-associative algebras correspond to order 3 elements. This suggests a potential group $G$, whose action between these algebras may  transform pre-Lie structures of one class to another, with algebras belonging to the same conjugacy class of $S_3$, transforming among themselves. Since pre-Lie structures geometrically realize connections on Lie group manifolds, the proposed group action realizes a geometric transformation between connections. In mathematical physics, one would recognize these as gauge transformations on the principal $G$-bundle over the Lie group manifold, which in this case happens to be a non-commutative base-space. This suggests interesting classes of gauge theories on non-commutative group manifold geometries. Other approaches for extracting geometric and pregeometric data from purely algebraic constructions can be found in \cite{kauffman2022calculus, arsiwalla2024operator, chester2024quantization, arsiwalla2024pregeometric, arsiwalla2023pregeometry, rickles2025quantum}.

Another direction for future work involves going beyond Lie algebras over ${\mathbb C}$. One may ask what kinds of pre-Lie structures do Lie algebras defined over noncommutative rings admit? In fact, Lie algebras over noncommutative rings were first introduced by Berenstein and Retakh \cite{berenstein2008lie, berenstein2005noncommutative}. However, this remains a largely uncharted terrain, and the study of pre-Lie structures admissible by such Lie algebras may shed important geometric insights about the associated Lie groups over noncommutative rings. The latter have recently been discussed in \cite{alessandrini2022symplectic, greenberg2024mathrm}. We shall address this issue in a forthcoming publication.   


\section*{Acknowledgments}

Thank you to the organizers of the amazing Wolfram Summer School 2024, where this project was initiated. The authors also thank Prof. Gil Salgado for many useful discussions and critical feedback on the original version of the manuscript.


\bibliographystyle{eptcs}
\bibliography{hottrefs.bib}

@article{kauffman2022calculus,
  title={Calculus, Gauge Theory and Noncommutative Worlds},
  author={Kauffman, Louis H},
  journal={Symmetry},
  volume={14},
  number={3},
  pages={430},
  year={2022},
  publisher={MDPI}
}

@incollection{arsiwalla2023pregeometry,
  title={Pregeometry, Formal Language, and Constructivist Foundations of Physics},
  author={Arsiwalla, Xerxes D and Elshatlawy, Hatem and Rickles, Dean},
  booktitle={Quantum Gravity and Computation},
  pages={137--159},
  publisher={Routledge},
  year={2025}
}

@book{rickles2025quantum,
  title={Quantum Gravity and Computation: Information, Pregeometry, and Digital Physics},
  author={Rickles, Dean and Arsiwalla, Xerxes D and Elshatlawy, Hatem},
  year={2025},
  publisher={Taylor \& Francis}
}

@article{arsiwalla2024pregeometric,
  title={Pregeometric spaces from wolfram model rewriting systems as homotopy types},
  author={Arsiwalla, Xerxes D. and Gorard, Jonathan},
  journal={International Journal of Theoretical Physics},
  volume={63},
  number={4},
  pages={83},
  year={2024},
  publisher={Springer}
}

@article{chester2024quantization,
  title={Quantization of a new canonical, covariant, and symplectic hamiltonian density},
  author={Chester, David and Arsiwalla, Xerxes D. and Kauffman, Louis H. and Planat, Michel and Irwin, Klee},
  journal={Symmetry},
  volume={16},
  number={3},
  pages={316},
  year={2024},
  publisher={MDPI}
}

@article{arsiwalla2024operator,
  title={On the operator origins of classical and quantum wave functions},
  author={Arsiwalla, Xerxes D. and Chester, David and Kauffman, Louis H.},
  journal={Quantum Studies: Mathematics and Foundations},
  volume={11},
  pages={193--215},
  year={2024},
  publisher={Springer}
}

@article{burde2006left,
  title={Left-symmetric algebras, or pre-Lie algebras in geometry and physics},
  author={Burde, Dietrich},
  journal={Central European Journal of Mathematics},
  volume={4},
  pages={323--357},
  year={2006},
  publisher={Springer}
}

@article{burde1998simple,
  title={Simple left-symmetric algebras with solvable Lie algebra},
  author={Burde, Dietrich},
  journal={Manuscripta Mathematica},
  volume={95},
  number={3},
  pages={397--411},
  year={1998},
  publisher={Springer}
}

@article{burde1994left,
  title={Left-symmetric structures on simple modular Lie algebras},
  author={Burde, Dietrich},
  journal={Journal of Algebra},
  volume={169},
  number={1},
  pages={112--138},
  year={1994},
  publisher={Elsevier}
}

@article{burde1999left,
  title={Left-invariant affine structures on nilpotent Lie groups},
  author={Burde, Dietrich},
  journal={Habilitation, D{\"u}sseldorf},
  year={1999}
}

@article{helmstetter1979radical,
  title={Radical d'une alg{\`e}bre sym{\'e}trique {\`a} gauche},
  author={Helmstetter, Jacques},
  journal={Annales de l'institut Fourier},
  volume={29},
  number={4},
  pages={17--35},
  year={1979}
}

@article{goze2004lie,
  title={Lie-admissible algebras and operads},
  author={Goze, Michel and Remm, Elisabeth},
  journal={Journal of algebra},
  volume={273},
  number={1},
  pages={129--152},
  year={2004},
  publisher={Elsevier}
}

@article{remm2002operades,
  title={Op{\'e}rades Lie-admissibles},
  author={Remm, Elisabeth},
  journal={Comptes rendus. Math{\'e}matique},
  volume={334},
  number={12},
  pages={1047--1050},
  year={2002}
}

@inproceedings{carter2012survey,
  title={A survey of quandle ideas},
  author={Carter, J Scott},
  booktitle={Introductory lectures on Knot Theory: Selected lectures presented at the Advanced School and Conference on Knot Theory and its Applications to Physics and Biology},
  pages={22--53},
  year={2012},
  organization={World Scientific}
}

@article{berenstein2008lie,
  title={Lie algebras and Lie groups over noncommutative rings},
  author={Berenstein, Arkady and Retakh, Vladimir},
  journal={Advances in Mathematics},
  volume={218},
  number={6},
  pages={1723--1758},
  year={2008},
  publisher={Elsevier}
}

@article{berenstein2005noncommutative,
  title={Noncommutative double Bruhat cells and their factorizations},
  author={Berenstein, Arkady and Retakh, Vladimir},
  journal={International Mathematics Research Notices},
  volume={2005},
  number={8},
  pages={477--516},
  year={2005}
}

@article{alessandrini2022symplectic,
  title={Symplectic groups over noncommutative algebras},
  author={Alessandrini, Daniele and Berenstein, Arkady and Retakh, Vladimir and Rogozinnikov, Eugen and Wienhard, Anna},
  journal={Selecta Mathematica},
  volume={28},
  number={4},
  pages={82},
  year={2022},
  publisher={Springer}
}

@article{greenberg2024mathrm,
  title={$SL_2$-like Properties of Matrices Over Noncommutative Rings and Generalizations of Markov Numbers},
  author={Greenberg, Zachary and Kaufman, Dani and Wienhard, Anna},
  journal={arXiv preprint arXiv:2402.19300},
  year={2024}
}

@article{cayley1857xxviii,
  title={On the theory of the analytical forms called trees},
  author={Cayley, Arthur},
  journal={The London, Edinburgh, and Dublin Philosophical Magazine and Journal of Science},
  volume={13},
  number={85},
  pages={172--176},
  year={1857},
  publisher={Taylor \& Francis}
}

@inproceedings{vinberg1963convex,
  title={Convex homogeneous cones},
  author={Vinberg, EB},
  booktitle={Transl. Moscow Math. Soc},
  volume={12},
  pages={340--403},
  year={1963}
}

@article{koszul1961domaines,
  title={Domaines born{\'e}s homogenes et orbites de groupes de transformations affines},
  author={Koszul, Jean-Louis},
  journal={Bulletin de la Soci{\'e}t{\'e} Math{\'e}matique de France},
  volume={89},
  pages={515--533},
  year={1961}
}

@article{santilli1968introduction,
  title={An introduction to Lie-admissible algebras},
  author={Santilli, RM},
  journal={Suppl. Nuovo Cimento},
  volume={6},
  pages={1225--1249},
  year={1968}
}

@article{milnor1977fundamental,
  title={On fundamental groups of complete affinely flat manifolds},
  author={Milnor, John},
  journal={Advances in Mathematics},
  volume={25},
  number={2},
  pages={178--187},
  year={1977},
  publisher={Elsevier}
}

@inproceedings{hounkonnou2016center,
  title={Center-symmetric algebras and bialgebras: relevant properties and consequences},
  author={Hounkonnou, Mahouton Norbert and Dassoundo, Mafoya Landry},
  booktitle={Geometric Methods in Physics: XXXIV Workshop, Bia{\l}owie{\.z}a, Poland, June 28--July 4, 2015},
  pages={281--293},
  year={2016},
  organization={Springer}
}

@article{bhandari1972classification,
  title={On the classification of simple antiflexible algebras},
  author={Bhandari, Mahesh Chandra},
  journal={Transactions of the American Mathematical Society},
  volume={173},
  pages={159--181},
  year={1972}
}

@article{rodabaugh1965generalization,
  title={A generalization of the flexible law},
  author={Rodabaugh, D},
  journal={Transactions of the American Mathematical Society},
  volume={114},
  number={2},
  pages={468--487},
  year={1965},
  publisher={JSTOR}
}

@article{rodabaugh1967semisimple,
  title={On semisimple antiflexible algebras},
  author={Rodabaugh, DJ},
  journal={Portugaliae mathematica},
  volume={26},
  pages={261--271},
  year={1967}
}

@article{rodabaugh1972antiflexible,
  title={On antiflexible algebras},
  author={Rodabaugh, David J},
  journal={Transactions of the American Mathematical Society},
  volume={169},
  pages={219--235},
  year={1972}
}

@article{myung1978nonflexible,
  title={Nonflexible Lie-admissible algebras},
  author={Myung, Hyo Chul},
  journal={Hadronic J.;(United States)},
  volume={1},
  number={3},
  year={1978},
  publisher={Univ. of Northern Iowa, Cedar Falls}
}

@article{akivis1976local,
  title={Local algebras of a multidimensional three-web},
  author={Akivis, Maks Aizikovich},
  journal={Siberian Mathematical Journal},
  volume={17},
  number={1},
  pages={3--8},
  year={1976},
  publisher={Springer}
}

\end{document}